\documentclass[a4paper,draft]{amsproc}
\usepackage{amssymb}
\usepackage{amscd} 
\usepackage{snapshot}
\usepackage[dvips]{graphicx}
\usepackage{color}
\usepackage{subfigure}
\usepackage{rotating}
\usepackage{amsmath}
\usepackage{amsthm}
\usepackage{amsfonts}

\theoremstyle{plain}
 \newtheorem{theorem}{Theorem}[section]
 
 \newtheorem{lemma}{Lemma}[section]
 \newtheorem{corollary}{Corollary}[section]
 \newtheorem{example}{Example}[section]
 
\theoremstyle{remark}
 
 \numberwithin{equation}{section}

\renewcommand{\leq}{\leqslant}
\renewcommand{\geq}{\geqslant}

\title[On order bounded subsets of locally solid Riesz spaces]{On order bounded subsets of locally solid Riesz spaces}



\author[Hong]{\bfseries Liang Hong}

\address{
Department of Mathematics \\ 
Robert Morris University   \\ 
Moon, PA 15108, USA}
\email{hong@rmu.edu}

\begin{document}

\vspace{18mm}
\setcounter{page}{1}
\thispagestyle{empty}


\begin{abstract}
In a topological Riesz space there are two types of bounded subsets: order bounded subsets and topologically bounded subsets.
It is natural to ask (1) whether an order bounded subset is topologically bounded and (2) whether a topologically bounded subset is order bounded. A classical result gives a partial answer to (1) by saying that an order bounded subset of a locally solid Riesz space is topologically bounded. This paper attempts to further investigate these two questions. In particular, we show that (i) there exists a non-locally solid topological Riesz space in which every order bounded subset is topologically bounded; (ii) if a topological Riesz space is not locally solid, an order bounded subset need not be topologically bounded; (iii) a topologically bounded subset need not be order bounded even in a locally convex-solid Riesz space. Next, we show that (iv) if a locally solid Riesz space has an order bounded topological neighborhood of zero, then every topologically bounded subset is order bounded; (v) however, a locally convex-solid Riesz space may not possess an order bounded topological neighborhood of zero even if every topologically bounded subset is order bounded; (vi) a pseudometrizable locally solid Riesz space need not have an order bounded topological neighborhood of zero. In addition, we give some results about the relationship between order bounded subsets and positive homogeneous operators.
\end{abstract}

\maketitle

\section{Notation and basic concepts}  
For notation, terminology and standard results concerning topological vector spaces, we refer to \cite{B1}, \cite{H1}, \cite{NB} and \cite{Schaefer}; for notation, terminology and standard results concerning Riesz spaces, we refer to \cite{AB1}, \cite{AB2} and \cite{LZ}. A partially ordered set $X$ is called a \emph{lattice} if the infimum and supremum of any pair of elements in $X$ exist. A real vector space $X$ is called an \emph{ordered vector space} if its vector space structure is compatible with the order structure in a manner such that
\begin{enumerate}
  \item [(a)]if $x\leq y$, then $x+z\leq y+z$ for any $z\in X$;
  \item [(b)]if $x\leq y$, then $\alpha x\leq \alpha y$ for all $\alpha\geq 0$.
\end{enumerate}
An ordered vector space is called a \emph{Riesz space} (or a \emph{vector lattice}) if it is also a lattice at the same time. A subset $Y$ of a Riesz space $X$ is said to be \emph{solid} if $|x| \leq |y|$ and $y\in Y$ imply that $x\in Y$. A subset $A$ of a topological vector space $(L, \tau)$ is said to be \emph{topologically bounded} or \emph{$\tau$-bounded} if for every neighborhood $V$ of zero there exists some $\lambda>0$ such that $\lambda E\subset V$. A topological vector space $(L, \tau)$ is said to be \emph{locally convex} if it has a neighborhood base at zero consisting of convex sets. A locally convex topological vector space is said to be \emph{seminormable} (\emph{normable}) if it can be generated by a single seminorm (norm). An \emph{ordered topological vector space} is an ordered vector space equipped with a compatible vector topology. A \emph{topological Riesz space} is an ordered topological vector space which is a Riesz space at the same time. A vector topology $\tau$ on a Riesz space $L$ is said to be \emph{locally solid} if there exists a $\tau$-neighborhood base at zero consisting of solid sets. A \emph{locally solid Riesz space} is a Riesz space equipped with a locally solid vector topology. A seminorm $\rho$ on a Riesz space $L$ is called a \emph{Riesz seminorm} if $|x|\leq |y|$ implies $\rho(x)\leq \rho(y)$ for any two elements $x, y\in L$. A vector topology on a Riesz space is said to be \emph{locally convex-solid} if it is both locally solid and locally convex. A \emph{locally convex-solid Riesz space} is a Riesz space equipped with a locally convex-solid vector topology.

Let $T$ be a linear operator between two ordered topological vector spaces $(L_1, \tau_1)$ and $(L_2, \tau_2)$.
$T$ is said to be a \emph{positive operator} if it carries positive elements to positive elements; it is said to be \emph{positive homogeneous} if $T(\lambda x)=\lambda T(x)$ for all $x\in L_1$ and $\lambda>0$; it is said to be an \emph{order bounded operator} if it carries order bounded sets to order bounded sets; it is said to be   \emph{topologically continuous} if $T^{-1}(O)\in \tau_1$ for every open set $O\in \tau_2$; it is said to be \emph{order continuous} if the net $(T(x_{\alpha}))$ is order convergent for every order-convergent net $(x_{\alpha})$ in $L_1$.

The terminology ``locally solid Riesz space'' seems to first appear in \cite {N1}, although the space is studied in \cite{R} under a different name. A systematic account of the theory can be found in \cite{AB1}.
\cite{Wong1} and \cite{Wong2} give several important results regarding order topology in the terminology of \cite{Schaefer}. The recent paper \cite{K} carries a similar title, but its content is very different from ours and no overlap seems to exist.

\section{Order bounded subsets of locally solid Riesz spaces}
Recall that a subset in a Riesz space is said to be \emph{order bounded} if it is contained in some order interval. Clearly, a subset of an ordered bounded set is order bounded. The next theorem gives some basic properties of order bounded sets.
\begin{lemma}\label{lemma2.1}
Let $L$ be a Riesz space. Then the following statements hold.
\begin{enumerate}
  \item [(i)]An arbitrary intersection of ordered bounded sets is order bounded.
  \item [(ii)]A finite union of ordered bounded sets is order bounded.
  \item [(iii)]The algebraic sum of two ordered bounded sets is order bounded.
  \item [(iv)]A nonzero multiple of an order bounded set is order bounded.
   \item [(v)]If $A$ is an ideal in $L$ and $T :L\rightarrow L/A$ is the canonical projection, then
              $T$ maps an order bounded set to an order bounded set, i.e., $T$ is order bounded.
\end{enumerate}
\end{lemma}
\begin{proof}
We only prove (v) as (i)-(iv) are trivial. Since $A$ is an ideal of $L$, the canonical projection $T: L\rightarrow L/A$ is a Riesz homomorphism. Thus, $T$ is a positive operator. The conclusion then follows from the fact that every positive operator is order bounded.
\end{proof}
\noindent \textbf{Remark.} An arbitrary union of ordered bounded sets need not be order bounded as the following example shows.

\begin{example}\label{example2.1}
Consider $R$ equipped with the usual ordering. Then $R$ is a Riesz space. For each positive integer $n$, the interval $[-n, n]$ is order bounded. However, $R=\cup_n [-n, n]$ is not order bounded.
\end{example}

Recall a classical result about order bounded subsets in a locally solid Riesz space (Theorem 2.19 of \cite{AB1}).
\begin{theorem}\label{theorem2.0}
If $(L, \tau)$ is a locally solid Riesz space, then an order bounded subset of $L$ is $\tau$-bounded.
\end{theorem}

We point out that the converse of Theorem \ref{theorem2.0} does not hold, that is, there exists a topological Riesz space $(L, \tau)$ which is not locally solid but every order bounded subset of $L$ is $\tau$-bounded. The next example gives such a space.

\begin{example}\label{example2.0}
Let $K$ be a compact subset of $R$ and $L$ be the Riesz space of all continuous real-valued functions on $K$ under the ordering: $x\leq y$ if and only if $x(t)\leq y(t), \forall t\in K$. Equip $L$ with the supremum norm $||f||_{\infty}=\sup_{x\in K} |f(x)|$. Let $\tau$ be the norm topology generated by this norm. Then $(L, \tau)$ is a Banach lattice. Let $\sigma(L, L')$ be the weak topology on $L$. Since $L$ is infinite-dimensional, the topological Riesz space $(L, \sigma(L, L'))$ is not locally solid (Theorem 2.38 of \cite{AB1}). Since $(L, \tau)$ is normable, it is a Hausdorff locally convex space. Thus, $\tau$-bounded sets coincide with $\sigma(L, L')$-bounded sets in $L$ (Theorem 2.4 of \cite{AB1}). It follows that every order bounded sets in $L$ is $\sigma(L, L')$-bounded.
\end{example}

The next example shows that Theorem \ref{theorem2.0} is not true if the ordered topological vector space is not locally solid. It provides another example of a topological Riesz space which is not locally solid.

\begin{example}\label{example2.2}
Consider the ordered topological vector space $(L, \tau)$, where $L=R^2$ with the lexicographic ordering and $\tau$ is the usual topology. Take two points $x=(-1, 0)$ and $y=(1, 0)$ in $L$. Then the order interval $[x, y]$ is trivially order bounded. It is clear that $[x, y]$ contains uncountably many infinite vertical rays. Since $\tau$ is the usual topology, a $\tau$-bounded set cannot contain any infinite vertical ray. It follows that $[x,y]$ is not $\tau$-bounded. Note that $(L, \tau)$ is not locally solid.
\end{example}

It is natural to ask whether a $\tau$-bounded set is order bounded in a locally solid Riesz space $(L, \mathcal{T})$. The following example shows that the answer is negative even in a locally convex-solid Riesz space.

\begin{example}\label{example2.3}
Let $L$ be the space of all Lebesgue measurable functions on $R$ with the usual pointwise ordering, i.e., for $x, y\in L$, we define $x\leq y$ if and only if $x(t)\leq y(t)$ for every $t\in R$. Consider the map $||\cdot||: L\rightarrow R$ defined by $||x||=\left(\int x^2(t)dt\right)^{1/2}$, where $x\in L$. Then $||\cdot||$ is a seminorm. It is easy to see that it is also a Riesz seminorm. Thus, the topology $\tau$ generated by $||\cdot||$ is locally convex-solid (Theorem 2.25 of \cite{AB1}). Put $A=\{x\in L\mid ||x||\leq 1\}$. Then $A$ is evidently a $\tau$-bounded subset of $L$. We claim that $A$ is not order bounded. To see this, we proceed by contraposition as follows. If there exists $x, y\in L$ such that $A\subset [x, y]$, then take $z\in A$ and define
\begin{equation*}
\widetilde{z}=\left\{
                \begin{array}{ll}
                  z(t), & \hbox{if $t\neq 0$;} \\
                 y(0)+1, & \hbox{if $t=0$.}
                \end{array}
              \right.
\end{equation*}
It is clear that $\widetilde{z}\in A$ but $\widetilde{z}\not \in [x, y]$, contradicting the hypothesis that $A\subset [x, y]$.
\end{example}

In view of the above discussion, one would naturally like to seek a condition under which a $\tau$-bounded subset of $L$ will be order bounded. The next theorem gives such a condition.

\begin{theorem}\label{theorem2.2}
Suppose $(L, \tau)$ is an ordered topological vector space (in particular a locally solid Riesz space) that has an order bounded $\tau$-neighborhood of zero, then every $\tau$-bounded subset is order bounded.
\end{theorem}
\begin{proof}
Let $\mathcal{N}$ denote the $\tau$-neighborhood base at zero. By hypothesis, there exists $V\in \mathcal{N}$ such that $V$ is contained in some order interval $[u, v]$ of $L$, where $u, v\in L$. Suppose $A$ is $\tau$-bounded subset of $L$. Then there exists $\lambda>0$ such that $\lambda A\subset V$. It follows that $A$ is contained in some order interval of $L$, that is, $A$ is order bounded.
\end{proof}


\noindent \textbf{Remark 1.} If the hypothesis of Theorem \ref{theorem2.2} holds for a locally solid Riesz space $(L, \tau)$, then $L$ is pseudometrizable as Theorem \ref{theorem2.3} shall show. However, even in a seminormable locally solid Riesz space $(L, \tau)$, a $\tau$-bounded set may not be order bounded. This is clear from Example \ref{example2.3}.\\

\noindent \textbf{Remark 2.} The converse of Theorem \ref{theorem2.2} is false even in a locally convex-solid Riesz space. In words, if $(L, \tau)$ is a locally convex-solid Riezs space where every $\tau$-bounded set is order bounded, $L$ may not possess an order bounded $\tau$-neighborhood of zero. The next example gives such a locally convex-solid space. \\

\begin{example}\label{example2.3.0}
Let $L$ be the space of all real-valued continuous functions defined on $R$ and $\tau$ be the compact-open topology. For every compact subset $K\subset R$, define $\rho_K(x)=\sup|x(K)|, x\in L$. Then $\tau$ is generated by the seminorms $\{\rho_K(x)\mid \ \text{$K$ is a compact subset of $R$}\}$; hence $\tau$ is locally convex. Define an order on $L$ by stipulating that $x\leq y$ if and only if $x(t)\leq y(t), \forall t\in R$. Clearly, $L$ is a locally convex-solid space. It is known that $(L, \tau)$ has no $\tau$-bounded neighborhood (Example 6.1.7 of \cite{NB}). Then Theorem \ref{theorem2.0} shows that $L$ cannot have an order bounded $\tau$-neighborhood of zero. On the other hand, a subset $E$ of $L$ is $\tau$-bounded if and only if $\rho_K(E)$ is bounded for every compact subset $K$ of $R$ (p. 109 of \cite{H1}). Thus, it is easy to see that every $\tau$-bounded set in $L$ is order bounded.
\end{example}

Recall that the \emph{diameter} of a subset $E$ of a pseudometric space $(L, d)$ is defined to be $diam (E)=\sup\{d(x, y)\mid x, y\in E\}$. We say $E$ has \emph{finite diameter} if $diam (E)<\infty$.
\begin{theorem}\label{theorem2.3}
Let $(L, \tau)$ be a (Hausdorff) locally solid Riesz space. Then the following statements hold.
\begin{enumerate}
  \item [(i)]If there exists an order bounded $\tau$-neighborhood of zero, then $L$ is (metrizable) pseudometrizable. In this case, every order bounded subset of $L$ has finite diameter with respect to every compatible (metric) pseudometric $d$.
  \item [(ii)]If there exists an order bounded convex $\tau$-neighborhood of zero, then $L$ is (normable) seminormable. In this case, every order bound subset of $L$ has finite diameter with respect to every compatible (norm) seminorm $\rho$.
  \item [(iii)]If $E$ is a subset of $L$ and every countable subset of $E$ is order bounded, then $E$ is $\tau$-bounded.
  \item [(iv)]Every order bounded subspace of $L$ is contained in the $\tau$-closure of zero.
  \item[(v)]If $(M, \tau')$ is another ordered topological vector space and $T$ is a topologically continuous linear operator from $L$ to $M$, then $T$ maps an ordered-bounded set to a $\tau'$-bounded set. In particular, if $L$ also has an order bounded $\tau$-neighborhood, then $T$ is order bounded.
\end{enumerate}
\end{theorem}

\begin{proof}\noindent
\begin{enumerate}
  \item [(i)]Since $L$ is a (Hausdorff) locally solid Riesz space, every order bounded set in $L$ must be $\tau$-bounded.
  It follows that an order bounded neighborhood of zero is a $\tau$-bounded neighborhood of zero.
  Thus, $L$ has $\tau$-bounded neighborhood of zero; hence it is pseudometrizable.
   The second part follows from the fact that a topologically bounded set in a pseudometric topological vector space has finite diameter.
  \item [(ii)]The proof is analogous to that of (i).
  \item [(iii)]By hypothesis, every countable subset of $E$ is $\tau$-bounded. Thus $E$ is $\tau$-bounded too.
  \item [(iv)]Let $E$ be an order bounded subspace of $L$. Then $E$ is $\tau$-bounded. Therefore, it is contained in the $\tau$-closure of zero.

  \item[(v)]The continuity and linearity of $T$ imply that $T$ maps a $\tau$-bounded set to a $\tau'$-bounded set. Since an ordered-bounded set in $L$ is $\tau$-bounded, $T$ maps  ordered-bounded sets to $\tau'$-bounded sets. If in addition $L$ has an order bounded $\tau$-neighborhood, Theorem \ref{theorem2.2} shows that a $\tau'$-bounded set is order bounded in $M$. It follows that $T$ is order bounded.
\end{enumerate}
\end{proof}

\begin{corollary}\label{corollarhy2.2}
Let $(L, \tau)$ be a Hausdorff locally solid Riesz space. Then every order bounded subspace of $L$ is trivial.
\end{corollary}

\noindent \textbf{Remark 1.} The converse of (i) need not hold, that is, a pseudometrizable locally solid Riesz space may not possess an order bounded $\tau$-neighborhood of zero.
Indeed, even the extra assumption ``locally convex'' may not help. Example \ref{example2.3} illustrates this point. The point is that seminormability guarantees the existence of a $\tau$-bounded neighborhood of zero. However, such a $\tau$-bounded neighborhood need not be order bounded.\\

\noindent \textbf{Remark 2.} The converse of (v) is not true in general. Consider the following example.

\begin{example}\label{example2.8}
Let $L=M$ be the Hilbert space $l^2$ of square-summable sequences, $\tau'$ be the norm topology, and $\tau$ be the weak topology $\sigma(L, L')$, where $L'$ denotes the topological dual of $L$. Equip $M$ and $L$ with the usual pointwise ordering, that is, for $x, y\in L=M$, $x\leq y$ if and only if $x_n\leq y_n$ for all $n\geq 1$. Let $T=I$ be the identity map. Then $T$ is clearly order bounded. However, $T$ is not topologically continuous. To see this, let $<\cdot, \cdot>$ denote the inner product on $L$ and $e_n$ denote the $n$-th unit coordinate vector, i.e., $e_n\in l^2$, $n$-th coordinate of $e_n$ is $1$, and every other coordinate is zero. Then for every $f\in L'$, Riesz Representation Theorem implies that there exists $x_0\in L$ such that $f(e_n)=\langle e_n, x_0\rangle$. It follows from Bessel's inequality that the sequence $(f(e_n))$ converges to $0$, implying that the sequence $(e_n)$ converges to $0$ in the weak topology. But the sequence $(e_n)$ does not converge in the norm topology. This shows that $T$ is not topologically continuous.
\end{example}

However, we have the following partial converse of (v).

\begin{theorem}\label{theorem2.4}
Let $(L, \tau)$ and $(M, \tau')$ be two locally solid Riesz spaces and $T: L\rightarrow M$ be a positive homogeneous operator. If there exists a $\tau$-neighborhood $V$ of zero such that $T(V)$ is order bounded, then $T$ is topologically continuous.
\end{theorem}
\begin{proof}
Since $(M, \tau')$ is locally solid, $T(V)$ is $\tau'$-bounded by hypothesis.
Thus, for every $\tau'$-neighborhood $U$ there exists some $\lambda>0$ such that $\lambda T(V)\subset U$. It follows from the positive homogeneity of $T$ that $T(\lambda V)\subset U$. Since $U$ is arbitrary, this shows that $T$ is continuous.
\end{proof}

\begin{corollary}\label{corollary2.3}
Let $(L, \tau)$ be a locally solid Riesz space with an order bounded $\tau$-neighborhood, $(M, \tau')$ be a locally solid Riesz space, and $T$ be a positive homogeneous map that maps $\tau$-bounded sets to order bounded sets. Then $T$ is topologically continuous.
\end{corollary}

\begin{theorem}\label{theorem2.5}
Let $(L, \tau)$ and $(M, \tau')$ be two locally solid Riesz spaces. If $L$ has an order bounded $\tau$-neighborhood of zero and $M$ has an order bounded $\tau'$-neighborhood of zero, then every order bounded operator $T$ is topologically continuous.
\end{theorem}

\begin{proof}
Let $T : L\rightarrow M$ be an order bounded operator. By hypothesis and Theorem \ref{theorem2.2}, $T$ maps $\tau$-bounded sets in $L$ to $\tau'$-bounded sets in $M$. Since $L$ is pseudometrizable by Theorem \ref{theorem2.3}, $T$ is topologically continuous.
\end{proof}

Suppose $(L_i, \tau_{i}), (i=1, 2)$ are two ordered topological vector spaces and $T: L_1\rightarrow L_2$ is a linear operator. $T$ need not be order bounded even if it maps $\tau_1$-bounded sets to $\tau_2$-bounded set.
Consider the following example.

\begin{example}\label{example2.9}
Let $L_1$ be the locally convex-solid Riesz space in Example \ref{example2.3} and $L_2$ be $R$ with the usual topology and the usual ordering. Consider the linear operator $T$ from $(L_1, \tau_1)$ to $(L_2, \tau_2)$ given by $x\mapsto \int x(t)dt$. It is easy to see from Cauchy-Schwartz inequality that $T$ maps $\tau_1$-bounded sets to $\tau_2$-bounded sets. However, $T$ is not order bounded. To see this, consider the image of the order interval $[x, y]$ under $T$, where $x(t)\equiv 0$ and $y(t)\equiv 1$. Suppose $T([x, y])$ is order bounded. Then there exist $a, b\in R$ such that $T([x, y])\in [a, b]$. Take a positive integer $n$ such that $0<\frac{b}{n}\leq \frac{1}{2}$. Put
\begin{equation*}
z(t)=\left\{
       \begin{array}{ll}
         \frac{b}{n}, & \hbox{$0\leq t\leq n$;} \\
         1, & \hbox{$n<t\leq 2n$;} \\
         0, & \hbox{otherwise.}
       \end{array}
     \right.
\end{equation*}
Then $z$ belongs to the order interval $[x, y]$. But $T(z)\not \in [a, b]$, contradicting the hypothesis that $T([x, y])\in [a, b]$.
\end{example}

On the other hand, an order bounded operator $T$ between two ordered topological vector spaces need not map $\tau_1$-bounded sets to $\tau_2$-bounded sets. Consider the following example.

\begin{example}\label{example2.10}
Take $L_1=L_2=R^2$ with $\tau_1=\tau_2$ being the usual topology on $R^2$. Equip $L_1$ with the usual ordering and $L_2$ with the lexicographic ordering, respectively. Consider $T=I$, the identity map. It is clear that $T$ is order bounded. However, $T$ does not map $\tau_1$-bounded sets to $\tau_2$-bounded sets. To see this, take $A=[x, y]\subset L_1$, where $x=(-1, -1)$ and $y=(1, 1)$. Then $A$ is a bounded square in $L_1$; hence $A$ is $\tau_1$-bounded. However, $T(A)=A$ contains uncountably many infinite vertical rays; hence $T(A)$ cannot be $\tau_2$-bounded.
\end{example}

We conclude this section with an interesting result regarding the product of locally solid Riesz spaces.

\begin{theorem}\label{theorem2.6}
Let $\{(L_{\alpha}, \tau_{\alpha})\}_{\alpha\in A}$ be a family of nontrivial locally convex-solid spaces, where $A$ is an infinite set. Equip $L=\prod_{\alpha\in A}L_{\alpha}$ with the product topology $\tau$. If $L_{\alpha}$ has an order bounded $\tau_{\alpha}$ neighborhood of zero for all $\alpha\in A$, then $L$ has no order bounded $\tau$-neighborhood of zero.
\end{theorem}
\begin{proof}
By hypothesis, each $L_{\alpha}$ has an $\tau_{\alpha}$-bounded convex neighborhood of zero. Thus, $L_{\alpha}$ is seminormable. Since $A$ is an infinite set, $L$ is not seminormable; hence it has no topologically bounded neighborhood of zero in the product topology. It follows from Theorem \ref{theorem2.3} that $L$ has no order bounded neighborhood of zero.
\end{proof}

\noindent \textbf{Remark.} In view of Theorem \ref{theorem2.6}, one may raise the question: Is it possible for $L=\prod L_{\alpha\in A}$ to have an order bounded $\tau$-neighborhood at zero while some $L_{\alpha}$ does not possesses an order bounded  $\tau_{\alpha}$ neighborhood of zero? The answer is negative. To see this, we proceed by contraposition to suppose that $L$ has an order bounded $\tau$-neighborhood $V=\prod_{\alpha\in A} V_{\alpha}$ of zero. Then $V$ is $\tau$-bounded by Theorem \ref{theorem2.0}; hence $L$ is seminormable. On the other hand, let $\pi_{\alpha}$ be the projection from $L$ to $L_{\alpha}$. Then $V_{\alpha}=\pi_{\alpha}(V)$ is $\tau_{\alpha}$-bounded for each $\alpha\in A$ (p. 118 of \cite{H1}). It follows that each $L_{\alpha}$ is seminormable. Thus, $L$ cannot be seminormable.\\

Topological bornologies have been extensively studied in the theory of topological vector spaces. Readers may wonder whether there is a notion of order bornology in a topological Riesz space. The answer is affirmative. Interested readers may consult \cite{BV1} and \cite{BV2} for further details.

\section*{Acknowledgments}

\noindent Thanks are due to the anonymous referee for comments which led to several improvements in this article.
The author also thanks Gerard Buskes for pointing out some references and his encouragement.





\bibliographystyle{amsplain}

\begin{thebibliography}{n} 
\bibitem{AB1} Aliprantis, C.D. and Burkinshaw, O. (2003). \emph{Locally Solid Riesz Spaces with Applications to Economics}, Second Edition, American Mathematical Society, Providence, Rhode Island.

\bibitem{AB2} Aliprantis, C.D. and Burkinshaw, O. (1985). \emph{Positive Operators}, Springer, Berlin, Heidelberg, New York.

\bibitem{B1} Bourbaki, N. (1987). \emph{Elements of Mathematics: Topological Vector Spaces}, Chapters 1-5, Springer, Berlin, New York.

\bibitem{BV1} Buskes, G. and van Rooij, A. (2002). The bornological tensor product of two Riesz spaces, \emph{Ordered Algebraic Strutures, Dev. Math.}, 7, 3-9, Springer, Berlin, Heidelberg, New York.


\bibitem{BV2} Buskes, G. and van Rooij, A. (2002). The bornological tensor product of two Riesz spaces: proof and background material, \emph{Ordered Algebraic Strutures, Dev. Math.}, 7, 189-203, Springer, Berlin, Heidelberg, New York.



\bibitem{H1} Horv\'{a}th, J. (1966).\emph{ Topological Vector Spaces and Distributions}. Vol. I, Addison-Wesley, Reading, Massachusetts.

\bibitem{K} Khurana, S.S. (2008). Oder-bounded sets in locally solid Riesz spaces, \emph{Note di Matematica}. 1, 119-123.

\bibitem{LZ} Luxemburg, W.A.J. and Zaanen, A.C. (1971). \emph{Riesz Spaces, I}, North-Holland, Amsterdam.


\bibitem{N1} Namioka, I. (1957). Partially ordred linear topological spaces, \emph{Memoirs of the American Mathematical Society}, 24.


\bibitem{NB} Narici, L. and Beckenstein, E. (2011). \emph{Topological Vector Spaces}, Second Edition, CRC Press, Boca Raton.


\bibitem{R} Roberts, G.T. (1952). Topologies in vector lattices, \emph{Math. Proc. Cambridge Phil Soc.}, 48, 533-546.

\bibitem{Schaefer} Schaefer, H.H., (1974). \emph{Toploogical Vector Spaces}, Springer, Berlin, New York.


\bibitem{Wong1}Wong, Y.C. (1970). The order-topology on Riesz spaces, \emph{Math. Proc. Cambridge Phil Soc.}, 67, 587-593.

\bibitem{Wong2} Wong, Y.C. (1972). The order-bound topology, \emph{Math. Proc. Cambridge Phil Soc.}, 71, 321-327.


\end{thebibliography}

\end{document}